\documentclass{irmaart}
\usepackage[all]{xy}
\usepackage{makeidx}

\theoremstyle{plain}
\newtheorem{theorem}{Theorem}[section]
\newtheorem{corollary}[theorem]{Corollary}
\newtheorem{lemma}[theorem]{Lemma}
\newtheorem{proposition}[theorem]{Proposition}
\newtheorem{conjecture}[theorem]{Conjecture}

\theoremstyle{definition}
\newtheorem{definition}[theorem]{Definition}
\newtheorem{remark}[theorem]{Remark}
\newtheorem{example}[theorem]{Example}
\newtheorem{property}[theorem]{Property}

\newcommand{\N}{{\mathbb{N}}}     
\newcommand{\Q}{{\mathbb{Q}}}     
\newcommand{\R}{{\mathbb{R}}}

\newcommand{\kk}{{\mathbf{k}}}     
\newcommand{\id}{\operatorname{id}}         
     
\makeindex
\begin{document}

\address{Faculte des Sciences\\
2 Boulevard Lavoisier,
49045 Angers Cedex 01, France\\
email:\,\tt{luc.menichi@univ-angers.fr}
}

\title{Rational homotopy -- Sullivan models}

\author{Luc Menichi}

\maketitle

\begin{abstract}
This chapter is a short introduction to Sullivan models. In
particular, we find the Sullivan model of a free loop space and use
it to prove the Vigu\'e-Poirrier-Sullivan theorem on the Betti numbers of a free loop space.
\end{abstract}



\vskip 1cm
In the previous chapter, we have seen the following theorem due
to Gromoll and Meyer.
\begin{theorem}
Let $M$ be a compact simply connected manifold.
If the sequence of Betti numbers of the free loop space on $M$,
$M^{S^1}$, is unbounded then any Riemannian metric on $M$ carries
infinitely many non trivial and geometrically distinct closed geodesics.
\end{theorem}
In this chapter, using Rational homotopy, we will see exactly when the sequence of Betti numbers of
$M^{S^1}$ over a field of caracteristic $0$ is bounded (See Theorem~\ref{nombres de betti lacets libres pas bornes}
and its converse Proposition~\ref{monogene donne Betti bornes}). This was one of the first major applications of rational homotopy.

Rational homotopy associates to any rational simply connected space, a commutative differential graded algebra.
If we restrict to almost free commutative differential graded
algebras, that is "Sullivan models", this association is unique.
\section{Graded differential algebra}
\subsection{Definition and elementary properties}
All the vector spaces are over $\Q$ (or more generally over a field $\kk$
of characteric $0$). We will denote by $\mathbb{N}$ the set of
non-negative integers.

\begin{definition}
A (non-negatively upper)
\index{graded!vector space}%
\emph{graded vector space} $V$ is a family $\{V^n\}_{n\in\N}$ of vector spaces.
An element $v\in V_i$ is an element of $V$ of \emph{degree} $i$. The degree of $v$ is denoted $\vert v\vert$.
A {\it differential} $d$ in $V$ is a sequence of linear maps $d^n:V^n\rightarrow V^{n+1}$
such that $d^{n+1}\circ d^{n}=0$, for all $n\in\N$.
A differential graded vector space or \emph{complex} is a graded vector space equipped with a differential.
A morphism of complexes $f:V\buildrel{\simeq}\over\rightarrow W$ is a
\index{quasi-isomorphism}%
\emph{quasi-isomorphism} if the induced map in homology $H(f):H(V)\buildrel{\cong}\over\rightarrow H(W)$ is an isomorphism
in all degrees.
\end{definition}
\begin{definition}
\index{graded!algebra}%
A \emph{graded algebra}  is a graded vector space $A=\{A^n\}_{n\in\N}$, equipped with a multiplication
$\mu:A^p\otimes A^q\rightarrow A^{p+q}$.
The algebra $Â$ is \emph{commutative} if $ab=(-1)^{\vert a\vert\vert b\vert}ba$ for all $a$ and $b\in A$.
\end{definition}
\begin{definition}
\index{dga}%
  A differential graded algebra or \emph{dga} is a graded algebra equipped with a differential
$d:A^n\rightarrow A^{n+1}$ which is also a \emph{derivation}: this means that 
for $a$ and $b\in A$
$$
d(ab)=(da)b+(-1)^{\vert a\vert}a(db).
$$
\index{cdga}%
A \emph{cdga} is a commutative dga.
\end{definition}
\begin{example}\label{example cdga}
1) Let $(B,d_B)$ and $(C,d_C)$ be two cdgas.
Then the tensor product $B\otimes C$ equipped with the multiplication
$$(b\otimes c)(b'\otimes c'):=(-1)^{\vert c\vert\vert b'\vert} bb'\otimes cc' $$
and the differential
$$
d(b\otimes c)=(db)\otimes c+(-1)^{\vert b\vert}b\otimes dc.
$$
is a cdga.
The {\it tensor product of cdgas} is the sum (or coproduct) in the category of cdgas.

2) More generally, let $f:A\rightarrow B$ and $g:A\rightarrow C$ be two morphisms of cdgas.
Let $B\otimes_A C$ be the quotient of  $B\otimes C$ by
the sub graded vector spanned by elements of the form $bf(a)\otimes c-b\otimes g(a)c$,
$a\in A$, $b\in B$ and $c\in C$. Then $B\otimes_A C$ is a cgda such that the quotient map
$B\otimes C\twoheadrightarrow B\otimes_A C$ is a morphism of cdgas.
The cdga $B\otimes_A C$ is the pushout of $f$ and $g$ in the category of cdgas:
$$
\xymatrix{
A\ar[r]^f\ar[d]_g
& B\ar[d]\ar@/^/[ddr]\\
C\ar[r]\ar@/_/[drr]
&B\otimes_A C\ar@{.>}[dr]|-{\exists!}\\
&&D
}$$

3) Let $V$ and $W$ be two graded vector spaces.
We denote by $\Lambda V$ the free graded commutative algebra on $V$.

If $V=\Q v$, i. e.  is of dimension $1$ and generated by a single element $v$,
then 

-$\Lambda V$ is $E(v)=\Q\oplus \Q v$, the exterior algebra on $v$ if the degree of $v$ is odd
and  

-$\Lambda V$ is $\Q [v]=\oplus_{n\in\N} \Q v^n$, the polynomial or symmetric algebra on $v$ if the degree of $v$ is even.

Since $\Lambda$ is left adjoint to the forgetful functor from the category of commutative graded algebras
to the category of graded vector spaces, $\Lambda$ preserves sums:
there is a natural isomorphism of commutative graded algebras
$\Lambda (V\oplus W) \cong\Lambda V\otimes \Lambda W$.

Therefore $\Lambda V$ is the tensor product $E(V^{odd})\otimes S(V^{even})$ of the exterior algebra on the generators of odd
degree and of the polynomial algebra on the generators of even degree.
\end{example}
\begin{definition}
Let $f:A\rightarrow B$ be a morphism of commutative graded algebras.
Let $d:A\rightarrow B$ be a linear map of degree $k$.
By definition, $d$ is a \emph{$(f,f)$-derivation} 
if
for $a$ and $b\in A$
$$
d(ab)=(da)f(b)+(-1)^{k\vert a\vert}f(a)(db).
$$
\end{definition}
\begin{property}[Universal properties]\label{proprietes universelles}

1) Let $i_B:B\hookrightarrow B\otimes\Lambda V$, $b\mapsto b\otimes 1$
and $i_V:V\hookrightarrow B\otimes\Lambda V$, $v\mapsto 1\otimes v$
be the inclusion maps.
Let $\varphi:B\rightarrow C$ be a morphism of commutative graded algebras.
Let $f:V\rightarrow C$ be a morphism of graded vector spaces.
Then $\varphi$ and $f$ extend uniquely to a morphism
$B\otimes \Lambda V\rightarrow C$
of commutative
graded algebras such that the following diagram commutes
$$\xymatrix{
B\ar[r]^\varphi\ar[dr]_{i_B}
&C
&V\ar[l]_f\ar[dl]^{i_V}\\
& B\otimes \Lambda V\ar@{.>}[u]|-{\exists!}
}
$$

2) Let $d_B:B\rightarrow B$ be a derivation of degree $k$.
Let $d_V:V\rightarrow B\otimes\Lambda V$ be a linear map of degree $k$.
Then there is a unique derivation $d$ such that the following diagram
commutes.

$$\xymatrix{
B\ar[r]^-{i_B}
&B\otimes \Lambda V
&V\ar[l]_{d_V}\ar[dl]^{i_V}\\
B\ar[r]^-{i_B}\ar[u]^{d_B}& B\otimes \Lambda V\ar@{.>}[u]|-{\exists!d}
}
$$

3) Let $f:\Lambda V\rightarrow B$ be a morphism of commutative graded algebras.
Let $d_V:V\rightarrow B$ be a linear map of degree $k$. Then there exists a unique $(f,f)$-derivation $d$ extending $d_V$:
$$
\xymatrix{
V\ar[r]^{d_V}\ar[d]_{i_V}
&B\\
\Lambda V\ar@{.>}[ur]_{\exists!d}
}$$
\end{property}
\begin{proof}
1) Since $\Lambda V$ is the free commutative graded algebra on $V$, $f$ can be extended to a morphism of graded algebras $\Lambda V\rightarrow C$. Since the tensor product of commutative graded algebras is the sum
in the category of commutative graded algebras, we
obtain a morphism of commutative graded algebras from $B\otimes \Lambda V$
to $C$.

2) Since $b\otimes v_1\dots v_n$ is the product $(b\otimes 1)(1\otimes v_1)
\dots (1\otimes v_n)$, $d(b\otimes v_1\dots v_n)$ is given by
$$
d_B(b)\otimes v_1\dots v_n
+\sum_{i=1}^n (-1)^{k(\vert b\vert+\vert v_1\vert+\dots+\vert v_{i-1}\vert)} (b\otimes v_1\dots v_{i-1})(d_V v_i) (1\otimes v_{i+1}\dots
v_n)
$$

3) Similarly, $d(v_1\dots v_n)$ is given by
$$
\sum_{i=1}^n (-1)^{k(\vert v_1\vert+\dots+\vert v_{i-1}\vert)} f(v_1)\dots f(v_{i-1})d_V(v_i) f(v_{i+1})\dots
f(v_n)
$$
\end{proof}
\subsection{Sullivan models of spheres}\label{modeles de Sullivan des spheres}
{\bf Sullivan models of odd spheres $S^{2n+1}$, $n\geq 0$.}

Consider a cdga $A(S^{2n+1})$ whose cohomology is isomorphic as graded algebras
to the cohomology of $S^{2n+1}$ with coefficients in $\kk$:
$$
H^*(A(S^{2n+1}))\cong H^*(S^{2n+1}).
$$
When $\kk$ is $\R$, you can think of $A$ as  the De Rham algebra of forms on $S^{2n+1}$.
There exists a cycle $v$ of degree $2n+1$ in  $A(S^{2n+1})$ such that
$$
H^*(A(S^{2n+1}))=\Lambda [v].
$$
The inclusion of complexes $(\kk v,0)\hookrightarrow A(S^{2n+1})$
extends to a unique morphism of cdgas $m:(\Lambda v,0)\rightarrow
A(S^{2n+1})$(Property~\ref{proprietes universelles}):
 
 $$
\xymatrix{
(\kk v,0)\ar[r]\ar[d]
& A(S^{2n+1}) \\
(\Lambda v,0)\ar@{.>}[ur]_{\exists!m}
}$$
The induced morphism in homology $H(m)$ is an isomorphism.
We say that $m:(\Lambda v,0)\buildrel{\simeq}\over\rightarrow
A(S^{2n+1})$ is a Sullivan model of $S^{2n+1}$

\noindent {\bf Sullivan models of even spheres $S^{2n}$, $n\geq 1$.}

Exactly as above, we construct a morphism of cdga
$m_1:(\Lambda v,0)\rightarrow A(S^{2n})$.
But now,  $H(m_1)$ is not an isomorphism:

$H(m_1)(v)=[v]$. Therefore $H(m_1)(v^2)=[v^2]=[v]^2=0$.
Since $[v^2]=0$ in $H^*(A(S^{2n}))$, there exists an element $\psi\in A(S^{2n})$
of degree $4n-1$ such that $d\psi=v^2$.

Let $w$ denote another element of degree $4n-1$.
The morphism of graded vector spaces $\kk v\oplus \kk w\hookrightarrow A(S^{2n})$, mapping $v$ to $v$ and $w$ to $\psi$
extends to a unique morphism of commutative graded algebras $m:\Lambda(v,w)\rightarrow A(S^{2n})$
 (1) of Property~\ref{proprietes universelles}):
$$
\xymatrix{
\kk v\oplus \kk w
\ar[r]\ar[d]
& A(S^{2n}) \\
\Lambda(v,w)\ar@{.>}[ur]_{\exists!m}
}$$

The linear map of degree $+1$, $d_V:V:=\kk v\oplus \kk w\rightarrow \Lambda(v,w)$
mapping $v$ to $0$ and $w$ to $v^2$ extends to a unique derivation
$d:\Lambda(v,w)\rightarrow \Lambda(v,w)$ (2) of Property~\ref{proprietes universelles}).
$$
\xymatrix{
\kk v\oplus \kk w\ar[r]^{d_V}\ar[d]
&\Lambda(v,w)\\
\Lambda(v,w)\ar@{.>}[ur]_{\exists!d}
}$$
Since $d$ is a derivation of odd degree, $d\circ d$ (which is equal to $1/2[d,d]$)
is again a derivation. The following diagram commutes
$$
\xymatrix{
V\ar[r]^{d_V}\ar[d]
&\Lambda V\ar[r]^{d}
&\Lambda V\\
\Lambda V\ar[urr]_{d\circ d}\ar[ur]^{d}
}
$$
Since the composite $d\circ d_V$ is null, by unicity
(2) of Property~\ref{proprietes universelles}),  the derivation $d\circ d$ is also null.
Therefore $(\Lambda V,d)$ is a cdga. This is the general method to check that $d\circ d=0$.

Denote by $d_A$ the differential on $A(S^{2n})$. 
Let's check now that $d_A\circ m=m\circ d$.
Since $d_A$ and $d$ are both $(id,id)$-derivations, $d_A\circ m$ and $m\circ d$
are both $(m,m)$-derivations.

Since
$d_A(m(v))=d_A(v)=0=m(0)=m(d(v))$ and
$d_A(m(w))=d_A(\psi)=v^2=m(v^2)=m(d(w))$,
$d_A\circ m$ and $m\circ d$ coincide on $V$. Therefore by unicity (3) of Property~\ref{proprietes universelles}), 
$d_A\circ m=m\circ d$. Again, this method is general.
So finally, we have proved that $m$ is a morphism of cdgas.
Now we prove that $H(m)$ is an isomorphism, by checking that $H(m)$ sends a basis to a basis.
\section{Sullivan models}
\subsection{Definitions}
Let $V$ be a graded vector space.
Denote by $V^+=V^{\geq 1}$ the sub graded vector space of $V$ formed by the elements of $V$
of positive degrees: $V=V^0\oplus V^+$.
\begin{definition}
\index{Sullivan model!relative}%
A \emph{relative Sullivan model} (or
\index{cofibration}%
\emph{cofibration} in the category of cdgas)
is a morphism of cdgas of the form
$$
(B,d_B)\hookrightarrow (B\otimes\Lambda V,d), b\mapsto b\otimes 1
$$
where

$\bullet$ $H^0(B)\cong \kk$,

$\bullet$ $V=V^{\geq 1}$,

$\bullet$ and $V$ is the direct sum of graded vector spaces $V(k)$:
$$
\forall n, V^n=\bigoplus_{k\in\mathbb{N}} V(k)^n
$$
such that $d:V(0)\rightarrow B\otimes \kk$
and $d:V(k)\rightarrow B\otimes \Lambda(V(<k))$.
Here $V(<k)$ denotes the direct sum $V(0)\oplus\dots\oplus V(k-1)$.
\end{definition}
Let $k\in\N$. Denote by $\Lambda^k V$ the sub graded vector space of  $\Lambda V$
generated by elements of the form $v_1\wedge\dots\wedge v_k$, $v_i\in V$.
Elements of $\Lambda^k V$ have by definition \emph{wordlength} $k$.
For example $\Lambda V=\kk\oplus V\oplus  \Lambda^{\geq 2}V$ .
 \begin{definition}
\index{Sullivan model!minimal}
A relative Sullivan model $(B,d_B)\hookrightarrow (B\otimes\Lambda V,d)$ is \emph{minimal} if 
$d:V\rightarrow B^+\otimes \Lambda V+ B\otimes \Lambda^{\geq 2}V$.
A \emph{(minimal) Sullivan model} is a (minimal) relative Sullivan model of the form
$(B,d_B)=(\kk,0)\hookrightarrow (\Lambda V,d)$.
\end{definition}
\begin{example}~\cite[end of the proof of Lemma 23.1]{Felix-Halperin-Thomas:ratht} 
Let $(\Lambda V,d)$ be cdga such that $V=V^{\geq 2}$.
Then $(\Lambda V,d)$ is a Sullivan model.
\end{example}
\begin{proof}[proof assuming the minimality condition]~\cite[p. 144]{Felix-Halperin-Thomas:ratht}
Suppose that 
$d:V\rightarrow\Lambda^{\geq 2}V$. In this case,
the $V(k)$ are easy to define:
let $V(k):=V^k$ for $k\in N$.
Let $v\in V^k$. By the minimality condition, $dv$ is equal to a sum $\sum_i x_iy_i$
where the non trivial elements $x_i$ and $y_i$ are both of positive length and therefore both of degre
$\geq 2$. Since $\vert x_i\vert+\vert y_i\vert=\vert dv\vert=k+1$, both 
$x_i$ and $y_i$ are of degree less than k. Therefore $dv$ belongs to
$\Lambda(V^{<k})=\Lambda(V(<k))$.
\end{proof}
\begin{property}
The composite of relative Sullivan models is again a Sullivan relative model.
\end{property}
\begin{definition}
\index{Sullivan model!of a cdga}
Let $C$ be a cdga.
A (minimal) \emph{Sullivan model of} $C$ is a (minimal) Sullivan model
$(\Lambda V,d)$ such that there exists a quasi-isomorphism of cdgas
$(\Lambda V,d)\buildrel{\simeq}\over\rightarrow C$.

Let $\varphi:B\rightarrow C$ be a morphism of cdgas.
A (minimal) \emph{relative Sullivan model of} $\varphi$ is a
(minimal) relative Sullivan model
$(B,d_B)\hookrightarrow (B\otimes \Lambda V,d)$ such that $\varphi$
can be decomposed as the composite of the relative Sullivan model and
of a quasi-isomorphism of cdgas:
$$
\xymatrix{
B\ar[r]^\varphi\ar[dr]
& C\\
& B\otimes\Lambda V\ar[u]_\simeq
}$$
\end{definition}
\begin{theorem}
Any morphism $\varphi:B\rightarrow C$ of cdgas admits a
minimal relative Sullivan model if $H^0(B)\cong \kk$, $H^0(\varphi)$
is an isomorphism and $H^1(\varphi)$ is injective.
\end{theorem}
This theorem is proved in general by Proposition 14.3 and Theorem 14.9
of~\cite{Felix-Halperin-Thomas:ratht}. But in practice, if
$H^1(\varphi)$ is an isomorphism, we construct a minimal relative
Sullivan model, by induction on degrees as in Proposition 12.2.
of~\cite{Felix-Halperin-Thomas:ratht}.
\subsection{An example of  relative Sullivan model}\label{exemple de
  model relatif de Sullivan}
Consider the minimal Sullivan model of an odd sphere found in
section~\ref{modeles de Sullivan des spheres}

$$(\Lambda v,0)\buildrel{\simeq}\over\rightarrow
A(S^{2n+1}).$$
Assume that $n\geq 1$.
Consider the multiplication of $\Lambda v$: the morphism of cdgas
$$
\mu:(\Lambda v_1,0)\otimes (\Lambda v_2,0)\rightarrow (\Lambda v,0),
v_1\mapsto v, v_2\mapsto v.
$$
Recall that $v$, $v_1$ and $v_2$ are of degree $2n+1$.

\index{s like operator suspension}%
\noindent Denote by $sv$ an element of degree $\vert sv\vert =\vert s\vert+\vert
v\vert=-1+\vert v\vert$. The operator $s$ of degre $-1$ is called the
\emph{suspension}.

We construct now a minimal relative Sullivan model of $\mu$.
Define $d(sv)=v_2-v_1$.
Let $m:\Lambda(v_1,v_2,sv),d\rightarrow (\Lambda v,0)$
be the unique morphism of cdgas extending $\mu$ such that $m(sv)=0$.
$$
\xymatrix{
(\Lambda v_1,0)\otimes (\Lambda v_2,0)\ar[r]^-\mu\ar[dr]
& (\Lambda v,0)\\
& \Lambda(v_1,v_2,sv,d)\ar[u]_m
}$$
\begin{definition}
Let $A$ be a differential graded algebra such that $A^0=\kk$.
\index{indecomposables}%
The complex of indecomposables of $A$, denoted $Q(A)$, is the
quotient $A^+/\mu(A^+\otimes A^+)$.
\end{definition}
The complex of indecomposables of $(\Lambda v,0)$, $Q((\Lambda v,0))$,
is $(\kk v,0)$ while $$Q(\Lambda(v_1,v_2,sv,d))=(\kk v_1\oplus\kk v_2\oplus\kk sv,d(sv)=v_2-v_1).$$
The morphism of complexes $Q(m): (\kk v_1\oplus\kk v_2\oplus\kk sv,d(sv)=v_2-v_1)\rightarrow (\kk v,0)$ map $v_1$ to $v$, $v_2$ to $v$ and $sv$ to $0$.
It is easy to check that $Q(m)$ is a quasi-isomorphism of complexes.

By Proposition 14.13 of~\cite{Felix-Halperin-Thomas:ratht},
since $m$ is a morphism of cdgas between Sullivan model,
$Q(m)$ is a quasi-isomorphim of if and only if
$m$ is a quasi-isomorphism.

So we have proved that $m$ is a quasi-isomorphism and therefore
 $$(\Lambda v_1,0)\otimes (\Lambda v_2,0)\hookrightarrow\Lambda(v_1,v_2,sv,d)$$
is a minimal relative Sullivan model of $\mu$.
Consider the following commutative diagram of cdgas where the square is a
pushout
$$
\xymatrix{
&& \Lambda v,0\\
\Lambda (v_1,v_2),0\ar[urr]^\mu\ar[d]_\mu\ar[r]
& \Lambda(v_1,v_2,sv),d\ar[ur]_m^\simeq\ar[d]\\
\Lambda v,0\ar[r]
&\Lambda v,0\otimes_{\Lambda (v_1,v_2),0} \Lambda(v_1,v_2,sv),d
}
$$
It is easy to check that the cdga $\Lambda v,0\otimes_{\Lambda (v_1,v_2),0} \Lambda(v_1,v_2,sv),d$ is isomorphic to $\Lambda(v,sv),0$.
As we will explain later, we have computed in fact, the minimal Sullivan model
$\Lambda(v,sv),0$
of the free loop space $(S^{2n+1})^{S^1}$.
In particular, the cohomology algebra 
$H^*((S^{2n+1})^{S^1};\kk)$ is isomorphic to $\Lambda(v,sv)$.
We can deduce easily that for $p\in\N$,
$
\operatorname{dim} H^p((S^{2n+1})^{S^1})\leq 1
$.
So we have shown that the sequence of Betti numbers of the free loop space
on odd dimensional spheres is bounded.
\subsection{The relative Sullivan model of the multiplication}
\begin{proposition}~\cite[Example 2.48]{Felix-Oprea-Tanre:algmodgeom}\label{modele de Sullivan de la multiplication}
Let $(\Lambda V,d)$ be  a relative minimal Sullivan model with $V=V^{\geq 2}$ (concentrated in degrees $\geq 2$). Then the multiplication
$\mu: (\Lambda V,d)\otimes(\Lambda V,d)\twoheadrightarrow(\Lambda V,d)$
admits a minimal relative Sullivan model of the form
$(\Lambda V\otimes \Lambda V\otimes\Lambda sV,D)$.
\end{proposition}
\begin{proof}[Constructive proof]
We proceed by induction on $n\in\mathbb{N^*}$ to construct
quasi-isomorphisms of cdgas $\varphi_n:(\Lambda V^{\leq n}\otimes \Lambda V^{\leq n}\otimes\Lambda sV^{\leq n},D)
\buildrel{\simeq}\over\twoheadrightarrow(\Lambda V^{\leq n},d)$
extending the multiplication on $\Lambda V^{\leq n}$.

Suppose that $\varphi_n$ is constructed.
We now define $\varphi_{n+1}$ extending $\varphi_n$ and $\mu$, the multiplication on $\Lambda V$.
Let $v\in V^{n+1}$. Then $d(v)\in\Lambda^{\geq 2}(V^{\leq n})$
and $\varphi_n(dv\otimes 1\otimes 1-1\otimes dv\otimes 1)=0$.
Since $\varphi_n$ is a surjective quasi-isomorphism, by the long exact sequence associated
to a short exact sequence of complexes, $\text{Ker }\varphi_n$ is acyclic.
Therefore since $dv\otimes 1\otimes 1-1\otimes dv\otimes 1$ is a cycle, there exists
an element $\gamma$ of degree $n+1$ of $\Lambda V^{\leq n}\otimes \Lambda V^{\leq n}\otimes\Lambda sV^{\leq n}$ such that $D(\gamma)=dv\otimes 1\otimes 1-1\otimes dv\otimes 1$ and $\varphi_n(\gamma)=0$. For degree reasons, $\gamma$ is decomposable, i. e. has wordlength $\geq 2$.
We define $D(1\otimes 1\otimes sv)=v\otimes 1\otimes 1-1\otimes v\otimes 1-\gamma$ and $\varphi_{n+1}(1\otimes 1\otimes sv)=0$.
Since $D\circ D(1\otimes 1\otimes sv)=0$ and $d\circ\varphi_{n+1}(1\otimes 1\otimes sv)=\varphi_{n+1}\circ d(1\otimes 1\otimes sv)$, by Property~\ref{proprietes universelles}, the derivation $D$ is a differential on
$\Lambda V^{\leq n+1}\otimes \Lambda V^{\leq n+1}\otimes\Lambda sV^{\leq n+1}$
and the morphism of graded algebras $\varphi_{n+1}$ is a morphism of complexes.

The complex of indecomposables of $(\Lambda V^{\leq n+1}\otimes \Lambda V^{\leq n+1}\otimes\Lambda sV^{\leq n+1},D)$,
$$Q((\Lambda V^{\leq n+1}\otimes \Lambda V^{\leq n+1}\otimes\Lambda sV^{\leq n+1},D)$$
is $(V^{\leq n+1}\oplus V^{\leq n+1}\oplus sV^{\leq n+1},d)$ with differential
$d$ given by
$d(v'\oplus v"\oplus sv)=v\oplus -v\oplus 0$ for $v'$, $v"$ and $v\in
V^{\leq n+1}$.
Therefore it is easy to check that $Q(\varphi_{n+1})$ is a quasi-isomorphism.
So by Proposition 14.13 of~\cite{Felix-Halperin-Thomas:ratht},
$\varphi_{n+1}$ is a quasi-isomorphism.
Since $\gamma$ is of degree $n+1$ and $sV^{\leq n}$ is of degree $<n$, this relative Sullivan model is minimal.
We now define $\varphi:(\Lambda V\otimes \Lambda V\otimes\Lambda sV,D)
\twoheadrightarrow(\Lambda V,d)$ as
$$\displaystyle\lim_{\longrightarrow}\varphi_n=
\bigcup_{n\in\mathbb{N}}\varphi_n: \bigcup_{n\in\mathbb{N}} \left(\Lambda V^{\leq n}\otimes \Lambda V^{\leq n}\otimes\Lambda sV^{\leq n}\right)\rightarrow\bigcup_{n\in\mathbb{N}}\Lambda V^{\leq n}.$$
Since homology commutes with direct limits in the category of complexes~\cite[Chap 4, Sect 2, Theorem 7]{Spanier:livre},
$H(\varphi)=\displaystyle\lim_{\longrightarrow}H(\varphi_n)$ is an isomorphism.
\end{proof}
\section{Rational homotopy theory}
Let $X$ be a topological space. Denote by $S^*(X)$ the singular
cochains of $X$ with coefficients in $\kk$. The dga  $S^*(X)$ is
almost never commutative. Nevertheless, Sullivan, inspired by Quillen
proved the following theorem.
\begin{theorem}~\cite[Corollary 10.10]{Felix-Halperin-Thomas:ratht}\label{quasi-isos entre A_PL et les cochaines}
For any topological space $X$,
there exists two natural quasi-isomorphisms of dgas
$$
S^*(X)\buildrel{\simeq}\over\rightarrow D(X)\buildrel{\simeq}\over\leftarrow A_{PL}(X)
$$
\index{$A_{PL}$}
where $A_{PL}(X)$ is commutative.
\end{theorem}
\begin{remark}~\label{sur les reels formes de De Rham}
This cdga $A_{PL}(X)$ is called the algebra of \emph{polynomial differential forms}.
If $\kk=\R$ and $X$ is a smooth manifold $M$, you can think that  $A_{PL}(M)$ is the De Rham algebra of differential
forms on $M$, $A_{DR}(M)$~\cite[Theorem
11.4]{Felix-Halperin-Thomas:ratht}.
\end{remark}
\begin{definition}\cite[Definition 2.34]{Felix-Oprea-Tanre:algmodgeom}
Two topological spaces $X$ and $Y$ have the same
\emph{rational homotopy type} if
there exists a finite sequence of continuous applications
$$X\buildrel{f_0}\over\rightarrow Y_1\buildrel{f_1}\over\leftarrow Y_2
\dots 
Y_{n-1}\buildrel{f_{n-1}}\over\leftarrow Y_n\buildrel{f_{n}}\over\rightarrow Y
$$
such that the induced maps in rational cohomology
\begin{multline*}
H^*(X;\Q)\buildrel{H^*(f_0)}\over\leftarrow H^*(Y_1;\Q)\buildrel{H^*(f_1)}\over\rightarrow H^*(Y_2;\Q)
\dots
H^*(Y_{n-1}1;\Q)\\\buildrel{H^*(f_{n-1})}\over\rightarrow H^*(Y_n;\Q)\buildrel{H^*(f_{n})}\over\leftarrow H^*(Y;\Q)
\end{multline*}
are all isomorphisms. 
\begin{theorem}\label{modele minimal unique et groupes d'homotopie}
Let $X$ be a path connected topological space.

1) (Unicity of minimal Sullivan models~\cite[Corollary p. 191]{Felix-Halperin-Thomas:ratht}) Two minimal Sullivan models of $A_{PL}(X)$ are isomorphic.

2) Suppose that $X$ is simply connected and $\forall n\in\mathbb{N}$, $H_n(X;\kk)$ is finite dimensional.
Let $(\Lambda V,d)$ be a minimal Sullivan model of $X$.
Then~\cite[Theorem 15.11]{Felix-Halperin-Thomas:ratht} for all $n\in\mathbb{N}$, $V^n$ is isomorphic to
$\text{Hom}_\kk(\pi_n(X)\otimes_\mathbb{Z} \kk,\kk)\cong\text{Hom}_\mathbb{Z}(\pi_n(X),\kk)$.
In particular~\cite[Remark 1 p.208]{Felix-Halperin-Thomas:ratht}, $\text{Dimension } V^n=  \text{Dimension } \pi_n(X)\otimes_\mathbb{Z} \kk < \infty$.
\end{theorem}
\begin{remark}
The isomorphim of graded vector spaces between $V$ and $\text{Hom}_\kk(\pi_*(X)\otimes_\mathbb{Z} \kk,\kk)$ is natural in some sense~\cite[p. 75-6]{Felix-Oprea-Tanre:algmodgeom}
with respect to maps $f:X\rightarrow Y$.
The isomorphism behaves well also with respect to the long exact sequence associated to a (Serre) fibration (\cite[Proposition 15.13]{Felix-Halperin-Thomas:ratht}
or~\cite[Proposition 2.65]{Felix-Oprea-Tanre:algmodgeom}).
\end{remark}
\end{definition}
\begin{theorem}\cite[Proposition 2.35]{Felix-Oprea-Tanre:algmodgeom}\cite[p. 139]{Felix-Halperin-Thomas:ratht}
Let $X$ and $Y$ be two simply connected topological spaces
such that $H^n(X;\Q)$ and $H^n(Y;\Q)$ are finite dimensional for all $n\in \N$.
Let $(\Lambda V,d)$ be a minimal Sullivan model of $X$
and let $(\Lambda W,d)$ be a minimal Sullivan model of $Y$.
Then $X$ and $Y$ have the same rational homotopy type if and only if
$(\Lambda V,d)$ is isomorphic to $(\Lambda W,d)$ as cdgas.
\end{theorem}
\section{Sullivan model of a pullback}
\subsection{Sullivan model of a product}
Let $X$ and $Y$ be two topological spaces.
Let $p_1:X\times Y\twoheadrightarrow Y$ and $p_2:X\times Y\twoheadrightarrow X$
be the projection maps.
Let $m$ be the unique morphism of cdgas given by the universal
property of the tensor product (Example~\ref{example cdga} 1))
$$
\xymatrix{
& A_{PL}(Y)\ar[d]\ar@/^/[ddr]^{A_{PL}(p_2)}\\
A_{PL}(X)\ar[r]\ar@/_/[drr]_{A_{PL}(p_1)}
&A_{PL}(X)\otimes
A_{PL}(Y)\ar@{.>}[dr]|-{\exists!m}\\
&&A_{PL}(X\times Y).
}
$$
Assume that $H^*(X;\kk)$ or $H^*(Y;\kk)$ is finite dimensional in all degrees.
Then~\cite[Example 2, p. 142-3]{Felix-Halperin-Thomas:ratht} $m$ is a quasi-isomorphism.
Let $m_X:\Lambda V\buildrel{\simeq}\over\rightarrow A_{PL}(X)$ be a Sullivan model of $X$.
Let $m_Y:\Lambda W\buildrel{\simeq}\over\rightarrow A_{PL}(Y)$ be a Sullivan model of $Y$.
Then by K\"unneth theorem, the composite
$$
\Lambda V\otimes \Lambda W\buildrel{m_X\otimes m_Y}\over\rightarrow
A_{PL}(X)\otimes A_{PL}(Y) \buildrel{m}\over\rightarrow
A_{PL}(X\times Y)
$$
is a quasi-isomorphism of cdgas.
Therefore we have proved that ``the Sullivan model of a product is the tensor product
of the Sullivan models''.

\subsection{the model of the diagonal}\label{modele de la diagonale}
Let $X$ be a topological space such that $H^*(X)$ is finite dimensional in all degrees.
Denote by $\Delta:X\rightarrow X\times X$, $x\mapsto (x,x)$ the diagonal map of $X$. Using the previous paragraph, since $A_{PL}(p_1\circ \Delta)=A_{PL}(p_2\circ \Delta)=A_{PL}(\id)=\id$,
we have the commutative diagram of
cdgas.
$$
\xymatrix{
A_{PL}(X)\ar[r]\ar[dr]_{A_{PL}(p_1)}\ar@/_2pc/[ddr]_{\id}
&A_{PL}(X)\otimes A_{PL}(X)\ar[d]^{m}_\simeq
& A_{PL}(X)\ar[l]\ar[dl]^{A_{PL}(p_2)}\ar@/^2pc/[ddl]^{\id}\\
&A_{PL}(X\times X)\ar[d]^{A_{PL}(\Delta)}\\
&A_{PL}(X)
}
$$
Therefore the composite $A_{PL}(X)\otimes A_{PL}(X)\buildrel{m}\over\rightarrow
A_{PL}(X\times X)\buildrel{A_{PL}(\Delta)}\over\rightarrow A_{PL}(X)$
coincides with the multiplication $\mu: A_{PL}(X)\otimes A_{PL}(X)\rightarrow A_{PL}(X)$.
Therefore the following diagram of cdgas commutes
$$
\xymatrix{
A_{PL}(X)
&A_{PL}(X\times X)\ar[l]_{A_{PL}(\Delta)}\\
& A_{PL}(X)\otimes A_{PL}(X)\ar[ul]_{\mu}\ar[u]_{m}^\simeq\\
\Lambda V\ar[uu]^{m_X}_\simeq
&\Lambda V\otimes \Lambda V\ar[l]^{\mu}\ar[u]_{m_X\otimes m_X}^\simeq\\
}
$$
Here $m_X:\Lambda V\buildrel{\simeq}\over\rightarrow A_{PL}(X)$ denotes a Sullivan model of $X$. Therefore we have proved that ``the morphism modelling the diagonal map
is the multiplication of the Sullivan model''.

\subsection{Sullivan model of a fibre product}\label{Sullivan model d'un produit fibre}

Consider a pullback square in the category of topological spaces
$$
\xymatrix{
P\ar[r]^g\ar[d]_q
&E\ar[d]^p\\
X\ar[r]^f
&B
}
$$

where

$\bullet$ $p:E\rightarrow B$ is a (Serre) fibration between two topological spaces,

$\bullet$ for every $i\in\mathbb{N}$, $H^i(X)$ and $H^i(B)$ are finite dimensional,

$\bullet$  the topological spaces $X$ and $E$ are path-connected and
$B$ is simply-connected.

\noindent Since $p$ is a (Serre) fibration, the pullback map $q$ is also a (Serre) fibration.
Let $A_{PL}(B)\otimes\Lambda V$ be a relative Sullivan model of $A(p)$.
Consider the corresponding commutative diagram of cdgas
$$
\xymatrix{
&A_{PL}(B)\ar[r]^{A_{PL}(f)}\ar[d]\ar@/_2pc/[ddl]_{A_{PL}(p)}
&A_{PL}(X)\ar[d]\ar@/^2pc/[ddr]^{A_{PL}(q)}\\
&A_{PL}(B)\otimes\Lambda V\ar[r]\ar[dl]_m^\simeq
&A_{PL}(X)\otimes_{A_{PL}(B)}A_{PL}(B)\otimes\Lambda V\ar@{.>}[dr]|-{\exists!m'}\\
A_{PL}(E)\ar[rrr]^{A_{PL}(g)}
&&&A_{PL}(P)
}
$$
where the rectangle is a pushout and $m'$ is given by the universal property.
Explicitly, for $x\in A_{PL}(X)$ and $e\in A_{PL}(B)\otimes\Lambda V$,
$m'(x\otimes e)$ is the product of $A_{PL}(q)(x)$
and $A_{PL}(g)\circ m(e)$.

Since $A_{PL}(B)\hookrightarrow A_{PL}(B)\otimes\Lambda V$ is a relative Sullivan model,
the inclusion obtained via pullback
$A_{PL}(X)\hookrightarrow A_{PL}(X)\otimes_{A_{PL}(B)}(A_{PL}(B)\otimes\Lambda V,d)\cong
(A_{PL}(X)\otimes\Lambda V,d)$ is also a relative Sullivan model (minimal if
$A_{PL}(B)\hookrightarrow A_{PL}(B)\otimes\Lambda V$ is minimal).

By~\cite[Proposition 15.8]{Felix-Halperin-Thomas:ratht} (or for weaker hypothesis~\cite[Theorem 2.70]{Felix-Oprea-Tanre:algmodgeom}),
\begin{theorem}
The morphism of cdgas $m'$ is a quasi-isomorphism.
\end{theorem}
We can summarize this theorem by saying that:
``The push-out of a (minimal) relative Sullivan model of a fibration is a (minimal) relative Sullivan
model of the pullback of the fibration.''
\begin{proof}[Idea of the proof]
Since by~\cite[Lemma 14.1]{Felix-Halperin-Thomas:ratht}, $A_{PL}(B)\otimes\Lambda V$
is a ``semi-free'' resolution of $A_{PL}(E)$ as left $A_{PL}(B)$-modules, by definition of the differential
torsion product,
$$
\text{Tor}^{A_{PL}(B)}(A_{PL}(X),A_{PL}(E)):=H(A_{PL}(X)\otimes_{A_{PL}(B)}(A_{PL}(B)\otimes\Lambda V).
$$
By Theorem~\ref{quasi-isos entre A_PL et les cochaines} and naturality,
we have an isomorphim of graded vector spaces
$$
\text{Tor}^{A_{PL}(B)}(A_{PL}(X),A_{PL}(E))\cong \text{Tor}^{S^*(B)}(S^*(X),S^*(E)).
$$
The Eilenberg-Moore formula gives an isomorphism of graded vector spaces
$$\text{Tor}^{S^*(B)}(S^*(X),S^*(E))\cong H^*(P).$$
We claimed that the resulting isomorphism between the homology of $A_{PL}(X)\otimes_{A_{PL}(B)}(A_{PL}(B)\otimes\Lambda V)$ and $H^*(P)$ can be identified with $H(m)$. Therefore $m$ is a quasi-isomorphism.
\end{proof}
Instead of working with $A_{PL}$, we prefer usually to work at the level of Sullivan models.
Let $m_B:\Lambda B\buildrel{\simeq}\over\rightarrow A_{PL}(B)$ be a Sullivan model of $B$.
Let $m_X:\Lambda X\buildrel{\simeq}\over\rightarrow A_{PL}(X)$ be a Sullivan model of $X$.
Let $\varphi$ be a morphism of cdgas such the following diagram commutes
exactly

$$
\xymatrix{
A_{PL}(B)\ar[r]^{A_{PL}(f)}
&A_{PL}(X)\\
\Lambda B\ar[r]^{\varphi}\ar[u]^{m_B}_\simeq
&\Lambda X\ar[u]^{m_X}_\simeq
}
$$
Let $\Lambda B\hookrightarrow \Lambda B\otimes \Lambda V$ be a relative
Sullivan model of $A_{PL}(p)\circ m_B$.
Consider the corresponding commutative diagram of cdgas
\begin{equation}\label{diagram Sullivan model d'un produit fibre}
\xymatrix{
A_{PL}(B)\ar[dd]_{A_{PL}(p)}
&\Lambda B\ar[r]^{\varphi}\ar[d]\ar[l]_{m_B}^ \simeq
&\Lambda X\ar[d]\ar[r]^{m_X}_\simeq
&A_{PL}(X)\ar[dd]^{A_{PL}(q)}\\
&\Lambda B\otimes\Lambda V\ar[r]\ar[dl]_m^\simeq
&\Lambda X\otimes_{\Lambda B}(\Lambda B\otimes\Lambda V)\ar@{.>}[dr]|-{\exists!m'}\\
A_{PL}(E)\ar[rrr]^{A_{PL}(g)}
&&&A_{PL}(P)
}
\end{equation}
where the rectangle is a pushout and $m'$ is given by the universal property.
Then again, $\Lambda X\hookrightarrow \Lambda X\otimes_{\Lambda B}(\Lambda B\otimes\Lambda V)$
is a relative Sullivan model and the morphism of cdgas $m'$ is a quasi-isomorphism.
\vskip 1cm
The reader should skip the following remark on his first reading.
\begin{remark}\label{modele a homotopie pres}
1) In the previous proof, if the composites $m_X\circ \varphi$ and $A_{PL}(f)\circ m_B$ are not strictly equal then the map $m'$ is not well defined.
In general, the composites $m_X\circ \varphi$ and $A_{PL}(f)\circ m_B$ are only homotopic and the situation is more complicated: see part 2) of this remark.

2) Let $m_B:\Lambda B\buildrel{\simeq}\over\rightarrow A_{PL}(B)$ be a Sullivan model of $B$.
Let $m_X':\Lambda X'\buildrel{\simeq}\over\rightarrow A_{PL}(X)$ be a Sullivan model of $X$.
By the lifting Lemma of Sullivan models~\cite[Proposition 14.6]{Felix-Halperin-Thomas:ratht},
there exists a morphism of cdgas $\varphi':\Lambda B\rightarrow\Lambda
X'$ such that the following diagram commutes
only up to homotopy (in the sense of~\cite[Section 2.2]{Felix-Oprea-Tanre:algmodgeom})
$$
\xymatrix{
A_{PL}(B)\ar[r]^{A_{PL}(f)}
&A_{PL}(X)\\
\Lambda B\ar[r]^{\varphi'}\ar[u]^{m_B}_\simeq
&\Lambda X'.\ar[u]^{m_X'}_\simeq
}
$$
In general, this square is not strictly commutative.
Let $\Lambda B\hookrightarrow \Lambda B\otimes \Lambda V$ be a relative
Sullivan model of $A_{PL}(p)\circ m_B$.
Then there exists a commutative diagram of cdgas
$$
\xymatrix{
A_{PL}(X)\ar[r]^{A_{PL}(q)}
& A_{PL}(P)\\
\Lambda X\ar[r]\ar[u]^ \simeq\ar[d]^ \simeq
& \Lambda X\otimes_{\Lambda B} (\Lambda B\otimes \Lambda V)\ar[u]^ \simeq\ar[d]^ \simeq\\
\Lambda X'\ar[r]
& \Lambda X'\otimes_{\Lambda B} (\Lambda B\otimes \Lambda V)}$$
\end{remark}
\begin{proof}[Proof of part 2) of Remark~\ref{modele a homotopie pres}]
Let $\Lambda B\buildrel{\varphi}\over\hookrightarrow \Lambda X\buildrel{\theta}\over\rightarrow \Lambda X'$ be a relative Sullivan model of $\varphi'$.
Since the composites $m_{X}'\circ\theta\circ\varphi$ and $A_{PL}(f)\circ m_B$ are homotopic,
by the homotopy extension property~\cite[Proposition 2.22]{Felix-Oprea-Tanre:algmodgeom} of the relative Sullivan model $\varphi:\Lambda B\hookrightarrow \Lambda X$,
there exists a morphism of cdgas $m_X:\Lambda X\rightarrow A_{PL}(X)$ homotopic to $m_{X}'\circ\theta$ such that $m_X\circ \varphi=A_{PL}(f)\circ m_B$.
Therefore using diagram~(\ref{diagram Sullivan model d'un produit fibre}), we obtain the following commutative diagram of cdgas:
$$
\xymatrix{
A_{PL}(X)\ar[r]^{A_{PL}(q)}
& A_{PL}(P)
&A_{PL}(E)\ar[l]_{A_{PL}(g)}\\
\Lambda X\ar[r]\ar[u]^\simeq_{m_X}\ar[d]_\simeq^{\theta}
& \Lambda X\otimes_{\Lambda B} (\Lambda B\otimes \Lambda V)\ar[u]^ \simeq_{m'}\ar[d]_\simeq^{\theta\otimes_{\Lambda B} (\Lambda B\otimes \Lambda V)}
&\Lambda B\otimes \Lambda V\ar[u]^\simeq_{m}\ar[l]\\
\Lambda X'\ar[r]
& \Lambda X'\otimes_{\Lambda B} (\Lambda B\otimes \Lambda V).
}$$
Here, since $\theta$ is a quasi-isomorphism, the pushout morphism $\theta\otimes_{\Lambda B} (\Lambda B\otimes \Lambda V)$ along the relative Sullivan model
$\Lambda X\hookrightarrow \Lambda X\otimes_{\Lambda B}(\Lambda B\otimes\Lambda V)$ is also a quasi-isomorphism ~\cite[Lemma 14.2]{Felix-Halperin-Thomas:ratht}.
\end{proof}
\subsection{Sullivan model of  a fibration}\label{modele de Sullivan d'une fibration}
Let $p:E\rightarrow B$ be a (Serre) fibration with fibre $F:=p^{-1}(b_0)$.

$$
\xymatrix{
F\ar[r]^j\ar[d]
&E\ar[d]^p\\
b_0\ar[r]
&B
}
$$
Taking $X$ to be the point $b_0$, we can apply the results
of the previous section.
Let $m_B:(\Lambda V,d)\buildrel{\simeq}\over\rightarrow A_{PL}(B)$
be a Sullivan model of $B$.
Let $(\Lambda V,d)\hookrightarrow (\Lambda V\otimes\Lambda W,d)$
be a relative Sullivan model of $A_{PL}(p)\circ m_B$.

Since $A_{PL}(\{b_0\})$ is equal to $(\kk,0)$, there is
a unique morphism of cdgas $m'$ such that the following diagram commutes

$$
\xymatrix{
A_{PL}(B)\ar[r]^{A_{PL}(p)}
&A_{PL}(E)\ar[r]^{A_{PL}(j)}
&A_{PL}(F)\\
(\Lambda V,d)\ar[r]\ar[u]_{m_B}^\simeq
&(\Lambda V\otimes\Lambda W,d)\ar[r]\ar[u]_\simeq
&(k,0)\otimes_{(\Lambda V,d)}(\Lambda V\otimes\Lambda W,d)\ar[u]_{m'}
}$$
Suppose that the base $B$ is a simply connected space
and that the total space $E$ is path-connected.
Then by the previous section, the morphism of cdga's
$$
m':(k,0)\otimes_{(\Lambda V,d)}(\Lambda V\otimes\Lambda W,d)\cong
(\Lambda W,\bar{d})\buildrel{\simeq}\over\longrightarrow  A_{PL}(F)
$$
is a quasi-isomorphism:

``
The cofiber of a relative Sullivan model of a fibration
is a Sullivan model of the fiber of the fibration.''

Note that the cofiber of a relative Sullivan model is minimal if and only if the relative
Sullivan model is minimal.
\subsection{Sullivan model of free loop spaces}
Let $X$ be a simply-connected space.
Consider the commutative diagram of spaces
$$
\xymatrix{
X^{S^1}\ar[r]\ar[d]_{ev}
& X^I\ar[d]_{(ev_0,ev_1)}
& X\ar[dl]^{\Delta}\ar[l]^{\approx}_\sigma
\\
X\ar[r]_-{\Delta}
&X\times X
}
$$
where the square is a pullback.
Here $I$ denotes the closed interval $[0,1]$, $ev$, $ev_0$, $ev_1$ are
the evaluation maps and the homotopy equivalence $\sigma:X\buildrel{\approx}\over\rightarrow X^I$ is the inclusion of constant paths.
Let $m_X:\Lambda V\buildrel{\simeq}\over\rightarrow A_{PL}(X)$
be a minimal Sullivan model of $X$.
By Proposition~\ref{modele de Sullivan de la multiplication}, the multiplication $\mu:\Lambda V\otimes\Lambda V\rightarrow  \Lambda V$
admits a minimal relative Sullivan model of the form
$$
\Lambda V\otimes\Lambda V\hookrightarrow\Lambda V\otimes\Lambda V\otimes\Lambda sV.
$$
Since $\mu$ is a model of the diagonal (Section~\ref{modele de la diagonale})
and since
$\Delta=(ev_0,ev_1)\circ\sigma$,
we have the commutative rectangle of cdgas
$$
\xymatrix{
A_{PL}(X\times X)\ar[rr]^{A_{PL}((ev_0,ev_1))}
&& A_{PL}(X^I)\ar[r]^{A_{PL}(\sigma)}
& A_{PL}(X)\\
\Lambda V\otimes\Lambda V\ar[u]^{m_{X\times X}}_\simeq\ar[rr]
&& \Lambda V\otimes\Lambda V\otimes\Lambda sV\ar[r]_-\simeq
&\Lambda V\ar[u]_{m_{X}}^\simeq
}
$$
Since $\sigma$ is a homotopy equivalence,
$S^*(\sigma)$ is a homotopy equivalence of complexes and
in particular a quasi-isomorphim.
So by Theorem~\ref{quasi-isos entre A_PL et les cochaines} and naturality,
$A_{PL}(\sigma)$ is also a quasi-isomorphism.
Therefore, by the lifting property of relative Sullivan
models~\cite[Proposition 14.6]{Felix-Halperin-Thomas:ratht},
there exists a morphism of cdgas
$\varphi:\Lambda V\otimes\Lambda V\otimes\Lambda sV
\rightarrow A_{PL}(X^I)
$ such that, in the diagram of cdgas
$$
\xymatrix{
A_{PL}(X\times X)\ar[r]^{A_{PL}((ev_0,ev_1))}
& A_{PL}(X^I)\ar[r]^{A_{PL}(\sigma)}_\simeq
& A_{PL}(X)\\
\Lambda V\otimes\Lambda V\ar[u]^{m_{X\times X}}_\simeq\ar[r]
& \Lambda V\otimes\Lambda V\otimes\Lambda sV\ar[r]_-\simeq
\ar@{.>}[u]^{\varphi}_\simeq
&\Lambda V\ar[u]_{m_{X}}^\simeq
}
$$
the left square commutes exactly and the right square
commutes in homology.
Therefore $\varphi$ is also a quasi-isomorphism.
This means that
$$
\Lambda V\otimes\Lambda V\hookrightarrow\Lambda V\otimes\Lambda V\otimes\Lambda sV.
$$
 is a relative Sullivan model of the composite
$$
\Lambda V\otimes \Lambda V\buildrel{m_{X\times X}}\over\rightarrow
A_{PL}(X\times X)\buildrel{A_{PL}((ev_0,ev_1))}\over\longrightarrow
A_{PL}(X^I).
$$
Here diagram~(\ref{diagram Sullivan model d'un produit fibre})
specializes to the following commutative diagram of cdgas
\begin{equation}\label{diagram Sullivan model lacets libres}
\xymatrix{
&\Lambda V\otimes\Lambda V\ar[r]^{\mu}\ar[d]
&\Lambda V\ar[d]\ar[r]^{m_X}_\simeq
&A_{PL}(X)\ar[dd]^{A_{PL}(ev)}\\
&\Lambda V\otimes\Lambda V\otimes\Lambda sV\ar[r]\ar[dl]_\varphi^\simeq
&\Lambda V\otimes_{\Lambda V\otimes\Lambda V}\Lambda V\otimes\Lambda V\otimes\Lambda sV\ar[dr]_{\simeq}\\
A(X^I)\ar[rrr]
&&&A(X^{S^1})
}
\end{equation}
where the rectangle is a pushout. Therefore
$$\Lambda V\hookrightarrow
\Lambda V\otimes_{\Lambda V\otimes\Lambda V}\left(\Lambda V\otimes\Lambda V\otimes\Lambda sV\right)\cong (\Lambda V\otimes\Lambda sV,\delta)
$$
is a minimal relative Sullivan model of
$A_{PL}(ev)\circ m_X$.
\begin{corollary}\label{theoreme de Chen sur lacets libres}
Let $X$ be a simply-connected space. Then the free loop
space cohomology of $H^*(X^{S^1};\kk)$ with coefficients in a field
$\kk$ of characteristic $0$ is isomorphic to the Hochschild homology
of $A_{PL}(X)$, $HH_*(A_{PL}(X),A_{PL}(X))$.
\index{Hochschild homology}
\end{corollary}
Replacing $A_{PL}(X)$ by $A_{DR}(M)$
(Remark~\ref{sur les reels formes de De Rham}), this Corollary is a theorem
of Chen~\cite[3.2.3 Theorem]{Brylinski:loopchageo} when $X$ is a
smooth manifold $M$.
\begin{proof}
The quasi-isomorphism of cdgas $m_X:\Lambda V
\buildrel{\simeq}\over\rightarrow A_{PL}(X)$ induces an isomorphism
between Hochschild homologies $$HH_*(m_X,m_X):HH_*(\Lambda V,\Lambda V)\buildrel{\cong}\over\rightarrow HH_*(A_{PL}(X), A_{PL}(X)).$$
By~\cite[Lemma 14.1]{Felix-Halperin-Thomas:ratht},
$\Lambda V\otimes\Lambda V\otimes\Lambda sV$ is a semi-free resolution of
$\Lambda V$ as a $\Lambda V\otimes\Lambda V^{op}$-module.
Therefore the Hochschild homology $HH_*(\Lambda V,\Lambda V)$ can be defined
as  the homology of the cdga $(\Lambda V\otimes\Lambda sV,\delta)$.
We have just seen above that
$H(\Lambda V\otimes\Lambda sV,\delta)$ is isomorphic to the free loop space
cohomology $H^*(X^{S^1};\kk)$.
\end{proof}
We have shown that a Sullivan model of $X^{S^1}$ is of the form
$(\Lambda V\otimes \Lambda sV,\delta)$.
The following theorem of Vigu\'e-Poirrier and Sullivan gives a precise description
of the differential $\delta$.
\begin{theorem}(\cite[Theorem p. 637]{Vigue-Sullivan:homtcg} or~\cite[Theorem 5.11]{Felix-Oprea-Tanre:algmodgeom}\label{differentiel du modele de Sullivan des lacets libres})
Let $X$ be a simply connected topological space.
Let $(\Lambda V,d)$ be a minimal Sullivan model of $X$.
For all $v\in V$, denote by $sv$ an element of degree $\vert v\vert -1$.
Let $s:\Lambda V\otimes \Lambda sV\rightarrow \Lambda V\otimes \Lambda sV$ be the unique
derivation of (upper) degree $-1$ such that on the generators $v$, $sv$, $v\in V$, $s(v)=sv$
and $s(sv)=0$. We have $s\circ s=0$.
Then there exists a unique Sullivan model of $X^{S^1}$ of the form
$(\Lambda V\otimes \Lambda sV,\delta)$ such that $\delta\circ s+s\circ \delta=0$
on $\Lambda V\otimes \Lambda sV$.
\end{theorem}
\begin{remark}\label{modele de la fibration des lacets libres}
Consider the free loop fibration
$\Omega X\hookrightarrow X^{S^1}\buildrel{ev}\over\twoheadrightarrow X$.
Since $(\Lambda V,d)\hookrightarrow (\Lambda V\otimes \Lambda sV,\delta)$
is a minimal relative Sullivan model of $A_{PL}(ev)\circ m_X$, by Section~\ref{modele de Sullivan d'une fibration}, 
$$
{\Bbbk}\otimes_{(\Lambda V,d)}(\Lambda V\otimes \Lambda sV,\delta)\cong (\Lambda sV,\bar{\delta})
$$
is a minimal Sullivan model of $\Omega X$. Let $v\in V$. By
Theorem~\ref{differentiel du modele de Sullivan des lacets libres},
$\delta(sv)=-s\delta v=-sdv$.
Since $dv\in\Lambda^{\geq 2}V$, $\delta(sv)\in \Lambda^{\geq 1}V\otimes \Lambda^1 sV$.
Therefore $\bar{\delta}=0$.
Since $\Omega X$ is a $H$-space, this follows also from
Theorem~\ref{model H-space} and from the unicity of minimal Sullivan models (part 1) of Theorem~\ref{modele minimal unique et groupes d'homotopie}).
\end{remark}

\section{Examples of Sullivan models}
\subsection{Sullivan model of spaces with polynomial cohomology}
The following proposition is a straightforward generalisation~\cite[p. 144]{Felix-Halperin-Thomas:ratht}
of the Sullivan model of odd-dimensional spheres (see section~\ref{modeles de Sullivan des spheres}).
\begin{proposition}\label{modele de Sullivan polynomial cohomology}
Let $X$ be a path connected topological space such that its cohomology $H^*(X;\kk)$ is a free graded commutative algebra $\Lambda V$
(for example, polynomial).
Then a Sullivan model of $X$ is $(\Lambda V,0)$.
\begin{example}
Odd-dimensional spheres $S^{2n+1}$, complex or quartenionic Stiefel manifolds~\cite[Example 2.40]{Felix-Oprea-Tanre:algmodgeom} $V_k(\mathbb{C}^n)$ or $V_k(\mathbb{H}^n)$,
classifying spaces $BG$ of simply connected Lie groups~\cite[Example 2.42]{Felix-Oprea-Tanre:algmodgeom}, connected Lie groups $G$ as we will see in the
following section.
\end{example}
\end{proposition}
\subsection{Sullivan model of an $H$-space}
An \emph{$H$-space} is a pointed topological space $(G,e)$ equipped with a pointed continuous map $\mu:(G,e)\times (G,e)\rightarrow (G,e)$
such that the two pointed maps $g\mapsto \mu(e,g)$ and $g\mapsto \mu(g,e)$ are pointed homotopic to the identity map of $(G,e)$.
\begin{theorem}~\cite[Example 3 p. 143]{Felix-Halperin-Thomas:ratht}\label{model H-space}
Let $G$ be a path connected $H$-space such that $\forall n\in\mathbb{N}$, $H_n(G;\kk)$ is finite dimensional. Then

1) its cohomology $H^*(G;\kk)$ is a free graded commutative algebra $\Lambda V$,

2) $G$ has a Sullivan model of the form $(\Lambda V,0)$, that is with zero differential.
\end{theorem}
\begin{proof}
1) Let $A$ be $H^*(G;\kk)$ the cohomology of $G$. By hypothesis, $A$
is a connected commutative graded Hopf algebra (not necessarily associative).
Now the theorem of Hopf-Borel in caracteristic $0$~\cite[VII.10.16]{DoldA:lecat} says that $A$ is a free graded commutative algebra.

2) By Proposition~\ref{modele de Sullivan polynomial cohomology}, 1) and 2) are equivalent.
\end{proof}
\begin{example}
Let $G$ be a path-connected Lie group (or more generally a $H$-space
with finitely generated integral homology).
Then $G$ has a Sullivan model of the form $(\Lambda V,0)$.
By Theorem~\ref{modele minimal unique et groupes d'homotopie}, $V^n$ and $\pi_n(G)\otimes_\mathbb{Z}\kk$ have the same
dimension for any $n\in\mathbb{N}$.
Since $H_*(G;\kk)$ is of finite (total) dimension, $V$ and therefore
$\pi_*(G)\otimes_\mathbb{Z}\kk$ are concentrated in odd degrees.
In fact, more generally~\cite[Theorem 6.11]{Browder:torsionH-space},
$\pi_2(G)=\{0\}$.
Note, however that $\pi_4(S^3)=\mathbb{Z}/2\mathbb{Z}\neq \{0\}$.
\end{example}
\subsection{Sullivan model of projective spaces}
Consider the complex projective space $\mathbb{CP}^n$, $n\geq 1$.
The construction of the Sullivan model of $\mathbb{CP}^n$ is similar to the construction of the Sullivan model of $S^2=\mathbb{CP}^ 1$ done in section~\ref{modeles de Sullivan des spheres}:

The cohomology algebra $H^*(A_{PL}(\mathbb{CP}^n))\cong H^*(\mathbb{CP}^n)$ is the truncated polynomial algebra $\frac{\kk[x]}{x^{n+1}=0}$
where $x$ is an element of degree $2$.
Let $v$ be a cycle of $A_{PL}(\mathbb{CP}^n)$ representing $x:=[v]$. 
The inclusion of complexes $(\kk v,0)\hookrightarrow A_{PL}(\mathbb{CP}^n)$
extends to a unique morphism of cdgas $m:(\Lambda v,0)\rightarrow
A_{PL}(\mathbb{CP}^n)$(Property~\ref{proprietes universelles}).
Since $[v^{n+1}]=x^{n+1}=0$, there exists an element $\psi\in A_{PL}(\mathbb{CP}^n)$ of degree $2n+1$ such that $d\psi=v^{n+1}$.
 Let $w$ denote another element of degree $2n+1$.
 Let $d$ be the unique derivation of $\Lambda(v,w)$ such that $d(v)=0$ and $d(w)=v^{n+1}$.
 The unique morphism of graded algebras $m:(\Lambda(v,w),d)\rightarrow A_{PL}(\mathbb{CP}^n)$ such that $m(v)=v$ and $m(w)=\psi$, is a morphism of cdgas.
 In homology, $H(m)$ sends $1$, $[v]$, \dots, $[v^n]$ to $1$, $x$, \dots, $x^n$. Therefore $m$ is a quasi-isomorphism.
 
 More generally, let $X$ be a simply connected space such that $H^*(X)$ is a truncated polynomial algebra $\frac{\kk[x]}{x^{n+1}=0}$
where $n\geq 1$ and $x$ is an element of even degree $d\geq 2$.
Then the Sullivan model of $X$ is $(\Lambda(v,w),d)$ where $v$ is an element of degree $d$, $w$ is an element of degree $d(n+1)-1$, $d(v)=0$ and $d(w)=v^{n+1}$.
\subsection{Free loop space cohomology for even-dimensional spheres and projective spaces}
In this section, we compute the free loop space cohomology of any simply connected space $X$ whose cohomology is a truncated polynomial algebra $\frac{\kk[x]}{x^{n+1}=0}$
where $n\geq 1$ and $x$ is an element of even degree $d\geq 2$.

Mainly, this is the even-dimensional sphere $S^d$ ($n=1$), the complex
projective space $\mathbb{CP}^n$ ($d=2$), the quaternionic projective space
$\mathbb{HP}^n$ ($d=4$) and the Cayley plane $\mathbb{OP}^2$ ($n=2$
and $d=8$).

In the previous section, we have seen that the minimal Sullivan model of $X$ is $(\Lambda(v,w),d(v)=0,d(w)=v^{n+1})$ where $v$ is an element of degree $d$ and $w$ is an element of degree $d(n+1)-1$.
By the constructive proof of Proposition~\ref{modele de Sullivan de la multiplication}, the multiplication $\mu$ of this minimal Sullivan model
$(\Lambda(v,w),d)$ admits the relative Sullivan model
$(\Lambda(v,w)\otimes \Lambda(v,w)\otimes \Lambda(sv,sw),D)$
where $$D(1\otimes 1\otimes sv)=v\otimes 1\otimes 1-1\otimes v\otimes 1\text{ and}$$
$$D(1\otimes 1\otimes sw)=w\otimes 1\otimes 1-1\otimes w\otimes 1-\sum_{i=0}^n v^i\otimes v^{n-i}\otimes sv.$$

Therefore, by taking the pushout along $\mu$ of this relative Sullivan model (diagram~(\ref{diagram Sullivan model lacets libres})),
or simply by applying Theorem~\ref{differentiel du modele de Sullivan des lacets libres}, a relative Sullivan model of $A_{PL}(ev)\circ m_X$
is given by the inclusion of cdgas
$
(\Lambda(v,w),d)\hookrightarrow (\Lambda(v,w,sv,sw),\delta)
$
where $\delta(sv)=-sd(v)=0$ and $\delta(sw)=-s(v^{n+1})=-(n+1)v^nsv$.
Consider the pushout square of cdgas
$$
\xymatrix{
(\Lambda(v,w),d)\ar[r]\ar[d]^\theta_\simeq
& (\Lambda(v,w,sv,sw),\delta)\ar[d]_\simeq^{\theta\otimes_{\Lambda (v,w)}\Lambda (sv,sw)} \\
(\frac{\kk[v]}{v^{n+1}=0},0)\ar[r]
& \left(\frac{\kk[v]}{v^{n+1}=0}\otimes \Lambda (sv,sw),\bar{\delta}\right).
}
$$
Here, since $\theta$ is a quasi-isomorphism, the pushout morphism $\theta\otimes_{\Lambda (v,w)}\Lambda (sv,sw)$ along the relative
Sullivan model $
\Lambda(v,w)\hookrightarrow \Lambda(v,w,sv,sw)
$
is also a quasi-isomorphism~\cite[Lemma 14.2]{Felix-Halperin-Thomas:ratht}.
Therefore, $H^*(X^{S^1};\kk)$ is the graded vector space
\begin{equation*}
{\Bbbk}\oplus \bigoplus_{1\leq p\leq n,\; i\in\mathbb{N}} {\Bbbk}v^{p}(sw)^i
\oplus \bigoplus_{0\leq p\leq n-1,\; i\in\mathbb{N}} {\Bbbk}v^{p}sv(sw)^i.
\end{equation*}
(In~\cite[Section 8]{MenichiL:cohrfl}, the author extends these
rational computations over any commutative ring.)
Since for all $i\in\mathbb{N}$, the degree of $v(sw)^{i+1}$ is
strictly greater than the degree of $v^n(sw)^i$, the generators $1$,
$v^{p}(sw)^i$, $1\leq p\leq n$, $i\in\mathbb{N}$, have all distinct
(even) degrees.
Since for all $i\in\mathbb{N}$, the degree of $sv(sw)^{i+1}$ is
strictly greater than the degree of $v^{n-1}sv(sw)^i$, the generators
$v^{p}sv(sw)^i$, $0\leq p\leq n-1$, $i\in\mathbb{N}$, have also distinct
(odd) degrees.
Therefore, for all $p\in\mathbb{N}$, $\text{Dim
}H^p(X^{S^1};\kk)\leq 1$.

At the end of section~\ref{exemple de model relatif de Sullivan},
we have shown the same inequalities when $X$ is an odd-dimensional
sphere,
or more generally for a simply-connected space $X$ whose cohomology
$H^*(X;\kk)$ is an exterior algebra $\Lambda x$ on an odd degree
generator $x$.
Since every finite dimensional graded commutative algebra
generated by a single element $x$  is either $\Lambda x$ or
$\frac{\kk[x]}{x^{n+1}=0}$, we have shown 
the following proposition:
\begin{proposition}\label{monogene donne Betti bornes}
Let $X$ be a simply connected topological space
such that its  cohomology $H^*(X;\kk)$ is generated by a single element and is finite  dimensional.
Then the
sequence of Betti numbers of the free loop space on $X$, $b_n:=\text{dim } H^n(X^{S^1};\kk)$
is bounded.
\end{proposition}
The goal of the following section will be to prove the converse of this proposition.
\section{Vigu\'e-Poirrier-Sullivan theorem on closed geodesics}
The goal of this section is to prove
(See section~\ref{proofofViguePoirrierSullivantheorem})
the following theorem due to
Vigu\'e-Poirrier and Sullivan.
\subsection{Statement of Vigu\'e-Poirrier-Sullivan theorem and of its generalisations}
\begin{theorem}(\cite[Theorem p. 637]{Vigue-Sullivan:homtcg} or~\cite[Proposition 5.14]{Felix-Oprea-Tanre:algmodgeom}\label{nombres de betti lacets libres pas bornes})
Let $M$ be a simply connected topological space
such that the rational cohomology of $M$, $H^*(M;\mathbb{Q})$ is of finite (total) dimension
(in particular, vanishes in higher degrees).

If the cohomology algebra $H^*(M;\mathbb{Q})$ requires at least two generators then the
sequence of Betti numbers of the free loop space on $M$, $b_n:=\text{dim } H^n(M^{S^1};\mathbb{Q})$
is unbounded.
\end{theorem}
\begin{example}\label{Betti rationel sur le produit de spheres}(Betti numbers of $(S^3\times S^3)^{S^1}$ over
  $\mathbb{Q}$)

Let $V$ and $W$ be two graded vector spaces such 
$\forall n\in\mathbb{N}$, $V^n$ and $W^n$ are finite dimensional.
We denote by $$P_{V}(z):=\sum_{n=0}^{+\infty}(\text{Dim }V^n) z^n$$
the sum of the \emph{Poincar\'e serie} of $V$.\index{Poincar\'e!serie}
If $V$ is the cohomology of a
space
$X$, we denote $P_{H^*(X)}(z)$ simply by $P_{X}(z)$.
Note that $P_{V\otimes W}(z)$ is the product $P_V(z)P_W(z)$.
We saw at the end of section~\ref{exemple de
  model relatif de Sullivan} that
$H^*((S^3)^{S^1};\mathbb{Q})\cong \Lambda v\otimes \Lambda sv$
where $v$ is an element of degree $3$.
Therefore
$$
P_{(S^3)^{S^1}}(z)=(1+z^3)\sum_{n=0}^{+\infty}z^{2n}=\frac{1+z^3}{1-z^2}.
$$
Since the free loops on a product is the product of the free loops
$$
H^*((S^3\times S^3)^{S^1})\cong H^*((S^3)^{S^1})\otimes H^*((S^3)^{S^1}).
$$
Therefore, since $\displaystyle{\frac{1}{1-z^2}=\sum_{n=0}^{+\infty}
  (n+1) z^{2n}}$,
$$
P_{(S^3\times
  S^3)^{S^1}}(z)=\left(\frac{1+z^3}{1-z^2}\right)^2=1+2z^2+\sum_{n=3}^{+\infty}
(n-1) z^n.
$$
So the Betti numbers over $\mathbb{Q}$ of the free loop space on
$S^3\times S^3$, $b_n:=\text{Dim }H^n((S^3\times
S^3)^{S^1};\mathbb{Q})$
are equal to $n-1$ if $n\geq 3$. In particular, they are unbounded.
\end{example}
\begin{conjecture}\label{conjecture geodesiques fermees}
The theorem of Vigu\'e-Poirrier and Sullivan holds replacing $\mathbb{Q}$ by any field $\mathbb{F}$.
\end{conjecture}
\begin{example}(Betti numbers of $(S^3\times S^3)^{S^1}$ over
  $\mathbb{F}$)

The calculation of Example~\ref{Betti rationel sur le produit de
  spheres}
over $\mathbb{Q}$ can be extended over any field $\mathbb{F}$ as
follows:
Since $S^3$ is a topological group, the map $\Omega S^3\times
S^3\rightarrow (S^3)^{S^1}$, sending $(w,g)$ to the free loop $t\mapsto
w(t)g$,
is a homeomorphism.
Using Serre spectral sequence\index{Serre!spectral sequence}
(\cite[Proposition 17]{Serre:suitespectrale} or\cite[Chap 9. Sect 7. Lemma
3]{Spanier:livre})
or Bott-Samelson theorem (\cite[Corollary 7.3.3]{SelickP:introhomot}
or~\cite[Appendix 2 Theorem 1.4]{Husemoller:fibb}), the cohomology of
the pointed loops on $S^3$, 
$H^*(\Omega S^3)$ is again isomorphic (as graded vector spaces only!) to the polynomial algebra
$\Lambda sv$ where $sv$ is of degree $2$.
Therefore exactly as over $\mathbb{Q}$,
$H^*((S^3)^{S^1};\mathbb{F})\cong \Lambda v\otimes \Lambda
sv$ where $v$ is an element of degree $3$.
Now the same proof as in Example~\ref{Betti rationel sur le produit de
  spheres}
shows that the Betti numbers over $\mathbb{F}$ of the free loop space on
$S^3\times S^3$, $b_n:=\text{Dim }H^n((S^3\times
S^3)^{S^1};\mathbb{F})$
are again equal to $n-1$ if $n\geq 3$.
\end{example}
In fact, the  theorem of Vigu\'e-Poirrier and Sullivan is completely algebraic:
\begin{theorem}(\cite{Vigue-Sullivan:homtcg} when $\mathbb{F}=\mathbb{Q}$,
\cite[Theorem III p. 315]{Halperin-Vigue:homfls} over any field $\mathbb{F}$)\label{nombres de Betti homologie de Hochschild}
Let $\mathbb{F}$ be a field. Let $A$ be a cdga such that $H^{<0}(A)=0$,
$H^{0}(A)=\mathbb{F}$ and $H^*(A)$ is of finite (total) dimension.
If the algebra $H^*(A)$ requires at least two generators then the
sequence of dimensions of the Hochschild homology\index{Hochschild homology} of $A$, $b_n:=\text{dim } HH_{-n}(A,A)$
is unbounded.
\end{theorem}
Generalising Chen's theorem
(Corollary~\ref{theoreme de Chen sur lacets libres})
over any field $\mathbb{F}$,
Jones theorem~\cite{JonesJ:Cycheh}  gives the isomorphisms of vector spaces

$$ H^n(X^{S^1};\mathbb{F})\cong HH_{-n}(S^*(X;\mathbb{F}), S^*(X;\mathbb{F})), \quad n\in\mathbb{Z}$$
between the free loop space cohomology of $X$ and the Hochschild
homology\index{Hochschild homology}
of the algebra of singular cochains on $X$.
But since the algebra of singular cochains $S^*(X;\mathbb{F})$ is not
commutative, Conjecture~\ref{conjecture geodesiques fermees}
does not follow from Theorem~\ref{nombres de Betti homologie de Hochschild}.
\subsection{A first result of Sullivan}
In this section, we start by a first result of Sullivan whose simple proof illustrates the technics used
in the proof of Vigu\'e-Poirrier-Sullivan theorem.
\begin{theorem}~\cite{Sullivan:conftokyo}\label{cohomologie des lacets libres pas bornee}
Let $X$ be a simply-connected space such that $H^*(X;\mathbb{Q})$
is not concentrated in degree $0$ and $H^n(X;\mathbb{Q})$ is null for
$n$ large enough. Then on the contrary, $H^n(X^{S^1};\mathbb{Q})\neq 0$ for an infinite set of integers $n$.
\end{theorem}
\begin{proof}
Let $(\Lambda V,d)$ be a minimal Sullivan model of $X$.
Suppose that $V$ is concentrated in even degree. Then $d=0$.
Therefore $H^*(\Lambda V,d)=\Lambda V$ is either concentrated in degree $0$
or is not null for an infinite sequence of degrees.
By hypothesis, we have excluded theses two cases. Therefore
$\text{dim }V^{odd}\geq 1$.

Let $x_1$, $x_2$, \dots, $x_m$, $y$, $x_{m+1}$, .....
be a basis of $V$ ordered by degree where $y$ denotes the first generator
of odd degree ($m\geq 0$).
For all $1\leq i\leq m$, $dx_i\in\Lambda x_{<i}$.
But $dx_i$ is of odd degree and $\Lambda x_{<i}$ is concentrated in even
degre. So $dx_i=0$.
Since $dy\in \Lambda x_{\leq m}$, $dy$ is equal to a polynomial $P(x_1,\dots,x_m)$
which belongs to $\Lambda^{\geq 2}(x_1,\dots,x_m)$.

Consider  $(\Lambda V\otimes\Lambda sV,\delta)$, the Sullivan model of $X^{S^1}$,
 given by Theorem~\ref{differentiel du modele de Sullivan des lacets libres}.
We have $\forall 1\leq i\leq m$, $\delta(sx_i)=-sdx_i=0$
and $\delta (sy)=-sdy\in \Lambda^{\geq 1}(x_1,\dots,x_m)\otimes \Lambda^{1}(sx_1,\dots,sx_m)$.
Therefore, since $sx_1$,\dots,$sx_m$ are all of odd degree, $\forall p\geq 0$,
$$\delta (sx_1\dots sx_m(sy)^p)=\pm sx_1\dots sx_m p\delta(sy)(sy)^{p-1}=0.$$
For all $p\geq 0$, the cocycle $sx_1\dots sx_m(sy)^p$ gives a non trivial
cohomology class in $H^*(X^{S^1};\mathbb{Q})$, since by Remark~\ref{modele de la fibration des lacets libres}, the image of this cohomology class in $H^*(\Omega X;\mathbb{Q})\cong \Lambda V$ is different from
$0$.
\end{proof}
\subsection{Dimension of $V^{odd}\geq 2$}
In this section, we show the following proposition:
\begin{proposition}\label{au moins deux generateurs de degree impair}
Let $X$ be a simply connected space such that $H^*(X;\mathbb{Q})$
is of finite (total) dimension and requires at least two generators.
Let $(\Lambda V,d)$ be the minimal Sullivan model of $X$.
Then $\text{dim }V^{odd}\geq 2$. 
\end{proposition}
\begin{property}(Koszul complexes)\label{complexes de Koszul a une
    variable}
\index{Koszul complexes}
Let $A$ be a graded algebra.
Let $z$ be a central element of even degree of $A$ which is not a divisor of zero.
Then we have a quasi-isomorphism of dgas
$$
(A\otimes\Lambda sz,d)\buildrel{\simeq}\over\twoheadrightarrow A/z.A\quad a\otimes 1\mapsto a, a\otimes sz\mapsto 0,
$$
where $d(a\otimes 1)=0$ and  $d(a\otimes sz)=(-1)^{\vert a\vert}az$
 for all $a\in A$.
\end{property}
\begin{proof}[Proof of Proposition~\ref{au moins deux generateurs de
    degree impair} (following $(2)\Rightarrow (3)$ of p. 214 of~\cite{Felix-Oprea-Tanre:algmodgeom})]
As we saw in the proof of Theorem~\ref{cohomologie des lacets libres pas bornee},
there is at least one generator $y$ of odd degree, that is $\text{dim }V^{odd}\geq 1$.
Suppose that there is only one.
Let $x_1$, $x_2$, \dots, $x_m$, $y$, $x_{m+1}$,\dots be a basis of $V$ ordered by degree ($m\geq 0$).

First case: $dy=0$. If $m\geq 1$, $dx_1=0$.
If $m=0$, $dx_1\in\Lambda^{\geq 2}(y)=\{0\}$ and therefore again $dx_1=0$. Suppose that for $n\geq 1$, $x_1^n$ is
a coboundary. Then $x_1^n=d(yP(x_1,\dots))=yd(P(x_1,\dots))$
where $P(x_1,\dots)$ is a polynomial in the $x_i$'s. But this is impossible since $x_1^n$
does not belong to the ideal generated by $y$.
Therefore for all $n\geq 1$, $x_1^n$ gives a non trivial cohomology class in $H^*(X)$.
But $H^*(X)$ is finite dimensional.

Second case: $dy\neq 0$. In particular $m\geq 1$. Since $dy$ is a non zero polynomial,
$dy$ is not a zero divisor, so by Property~\ref{complexes de Koszul a une variable}, we have a quasi-isomorphism of
cdgas
$$\Lambda(x_1,\dots,x_m,y)\buildrel{\simeq}\over\twoheadrightarrow\Lambda(x_1,\dots,x_m)/(dy).$$
Consider the push out in the category of cdgas
$$
\xymatrix{
\Lambda(x_1,\dots,x_m,y)\ar[r]\ar[d]_\simeq
& \Lambda(x_1,\dots,x_m,y,x_{m+1},\dots),d\ar[d]\\
\Lambda(x_1,\dots,x_m)/(dy)\ar[r]
&  \Lambda(x_1,\dots,x_m)/(dy)\otimes \Lambda(x_{m+1},\dots),\bar{d}
}$$
Since $\Lambda(x_1,\dots,x_m)/(dy)\otimes \Lambda(x_{m+1},\dots)$ is concentrated in
even degrees, $\bar{d}=0$.
Since the top arrow is a Sullivan relative model and the left arrow is a quasi-isomorphism,
the right arrow is also a quasi-isomorphism
(\cite[Lemma 14.2]{Felix-Halperin-Thomas:ratht}, or more generally the category of cdgas over $\mathbb{Q}$
is a Quillen model category).
Therefore the algebra $H^*(X)$ is isomorphic to
$\Lambda(x_1,\dots,x_m)/(dy)\otimes \Lambda(x_{m+1},\dots)$.
If $m\geq 2$, $\Lambda(x_1,\dots,x_m)/(dy)$ and so $H^*(X)$ is infinite dimensional. 
If $m=1$, since $\Lambda x_1/(dy)$ is generated by only one generator, we must have another
generator $x_2$. But $\Lambda(x_1)/(dy)\otimes \Lambda(x_{2},\dots)$ is also infinite dimensional.
\end{proof}
\subsection{Proof of Vigu\'e-Poirrier-Sullivan theorem}\label{proofofViguePoirrierSullivantheorem}
\begin{lemma}~\cite[Proposition 4]{Vigue-Sullivan:homtcg}\label{elimination generateur degree pair}
Let $A$ be a dga over any field such that the multiplication by a cocycle $x$ of any degre
$A\rightarrow A$, $a\mapsto xa$ is injective (Our example will be $A=(\Lambda V,d)$ and $x$
a non-zero  element of $V$ of even degree such that $dx=0$).
If the Betti numbers $b_n=\text{dim } H^n(A)$ of $A$ are bounded then the Betti numbers
$b_n=\text{dim } H^n(A/xA)$ of $A/xA$ are also bounded.
\end{lemma}
\begin{proof}
Since $H^n(xA)\cong H^{n-\vert x\vert}(A)$, the short exact sequence of complexes
$$
0\rightarrow xA\rightarrow A\rightarrow A/xA\rightarrow 0
$$
gives the long exact sequence in homology
$$
\dots\rightarrow H^n(A)\rightarrow H^n(A/xA)\rightarrow H^{n+1-\vert x\vert}(A)\rightarrow\dots
$$
Therefore $\text{dim }H^n(A/xA)\leq \text{dim } H^n(A)+\text{dim } H^{n+1-\vert x\vert}(A)$
\end{proof}
\begin{proof}[Proof of Vigu\'e-Poirrier-Sullivan theorem (Theorem~\ref{nombres de betti lacets libres pas bornes})]
Let $(\Lambda V,d)$ be the minimal Sullivan model of $X$.
Let $(\Lambda V\otimes\Lambda sV,\delta)$ be the Sullivan model of $X^{S^1}$ given by
Theorem~\ref{differentiel du modele de Sullivan des lacets libres}.
From Proposition~\ref{au moins deux generateurs de degree impair},
we know that $\text{dim }V^{odd}\geq 2$.
Let $x_1$, $x_2$, \dots, $x_m$, $y$, $x_{m+1}$,\dots, $x_n$,
$z=x_{n+1}$, \dots 
be a basis of $V$ ordered by degrees
where $x_1$,\dots, $x_n$ are of even degrees
and $y$, $z$ are of odd degrees.
Consider the commutative diagram of cdgas where the three rectangles are push outs
$$
\xymatrix{
\Lambda(x_1,\dots, x_n)\ar[r]\ar[d]
&(\Lambda V,d)\ar[r]\ar[d]
&(\Lambda V\otimes\Lambda sV,\delta)\ar[d]\\
\mathbb{Q}\ar[r]
&\Lambda(y,z,\dots)\ar[r]\ar[d]
&(\Lambda(y,z,\dots)\otimes\Lambda sV,\bar{\delta})\ar[d]\\
&\mathbb{Q}\ar[r]
&(\Lambda sV,0)
}
$$
Note that by Remark~\ref{modele de la fibration des lacets libres},
the differential on $\Lambda sV$ is $0$.

For all $1\leq j\leq n+1$, $$\delta x_j=dx_j\in \Lambda^{\geq 2}(x_{<j},y)\subset
\Lambda^{\geq 1}(x_{<j})\otimes \Lambda y.$$
Therefore $$\delta (sx_j)=-s\delta x_j\in \Lambda x_{<j}\otimes\Lambda^1 sx_{<j}\otimes\Lambda y+\Lambda^{\geq 1}(x_{<j})\otimes \Lambda^1 sy.$$
Since $(sx_1)^2=\dots=(sx_{j-1})^2=0$, the product
$$sx_1\dots sx_{j-1}  \delta (sx_j)\in \Lambda^{\geq 1}(x_{<j})\otimes \Lambda^1 sy.$$
So $\forall 1\leq j\leq n+1$, $sx_1\dots sx_{j-1}  \bar{\delta} (sx_j)=0$.
In particular $sx_1\dots sx_{n}  \bar{\delta} (sz)=0$.
Similarly, since $dy\in\Lambda^{\geq 2} x_{\leq m}$,
$sx_1\dots sx_m\delta(sy)=0$ and so $sx_1\dots sx_n\bar{\delta}(sy)=0$.
By induction, $\forall 1\leq j\leq n$, $\bar{\delta}(sx_1\dots sx_j)=0$.
In particular, $\bar{\delta}(sx_1\dots sx_n)=0$.
So finally, for all $p\geq 0$ and all $q\geq 0$,
$\bar{\delta}(sx_1\dots sx_n(sy)^p(sz)^q)=0$.
The cocycles $sx_1\dots sx_n(sy)^p(sz)^q$, $p\geq 0$, $q\geq 0$, give
linearly independent cohomology classes in $H^*(\Lambda(y,z,\dots)\otimes\Lambda sV,\bar{\delta})$ since their images in $(\Lambda sV,0)$ are linearly independent.

For all $k\geq 0$, there is at least $k+1$ elements of the form $sx_1\dots sx_n(sy)^p(sz)^q$
in degree $\vert sx_1\vert+ \dots+\vert sx_n\vert+k\cdot\text{lcm}(\vert sy\vert,\vert sz\vert)$
(just take $p=i\cdot\text{lcm}(\vert sy\vert,\vert sz\vert)/\vert sy\vert$ and
$q=(k-i)\text{lcm}(\vert sy\vert,\vert sz\vert)/\vert sz\vert$ for $i$ between $0$ and $k$).
Therefore the Betti numbers of $H^*(\Lambda(y,z,\dots)\otimes\Lambda sV,\bar{\delta})$ are unbounded.

Suppose that the Betti numbers of $(\Lambda V\otimes\Lambda sV,\delta)$ are bounded.
Then by Lemma~\ref{elimination generateur degree pair} applied to $A=(\Lambda V\otimes\Lambda sV,\delta)$ 
and $x=x_1$, the Betti numbers of the quotient cdga $(\Lambda(x_2,\dots)\otimes\Lambda sV,\bar{\delta})$
are bounded.
By continuing to apply Lemma~\ref{elimination generateur degree pair} to $x_2$, $x_3$, \dots, $x_n$,
we obtain that the Betti numbers of the quotient cdga $(\Lambda(y,z,\dots)\otimes\Lambda sV,\bar{\delta}$
are bounded. But we saw just above that they are unbounded.
\end{proof}
\section{Further readings}
In this last section, we suggest some further readings that we find appropriate for the student.

In~\cite[Chapter 19]{Bott-Tu:difforms}, one can find a very short and
gentle introduction to rational homotopy that the reader should
compare to our introduction.

In this introduction, we have tried to explain that rational homotopy
is a functor which transforms homotopy pullbacks of spaces into
homotopy pushouts of cdgas.
Therefore after our introduction,
we advise the reader to look at~\cite{Hess:introrationalhtpy}, a more advanced
introduction to rational homotopy, which explains the
model category of cdgas.

The canonical reference for rational
homotopy~\cite{Felix-Halperin-Thomas:ratht} is highly readable.

In the recent book~\cite{Felix-Oprea-Tanre:algmodgeom},
you will find many geometric applications of rational homotopy.
The proof of Vigu\'e-Poirrier-Sullivan theorem  we give here, follows more or less the proof given
in~\cite{Felix-Oprea-Tanre:algmodgeom}.

We also like~\cite{TanreD:homrmc} recently reprinted
because it is the only book where you can find the Quillen model
of a space: a differential graded Lie algebra representing its
rational homotopy type (instead of a commutative
differential graded algebra as the Sullivan model).
\bibliography{Bibliographie.bib}
\bibliographystyle{amsplain}
\printindex
\end{document}